\documentclass[a4paper,11pt]{article}

\usepackage[T1]{fontenc}
\usepackage[utf8]{inputenc}
\usepackage{amsmath,amssymb,amsthm,mathtools}
\usepackage{geometry}
\usepackage{hyperref}
\usepackage{authblk}
\newcommand{\C}{\mathbb{C}}
\newcommand{\ket}[1]{\lvert #1\rangle}
\newcommand{\QLS}{\mathrm{QLS}}
\newcommand{\card}{\operatorname{card}}
\newcommand{\vectset}{\mathcal{V}}

\newtheorem{definition}{Definition}
\newtheorem{proposition}{Proposition}
\newtheorem{remark}{Remark}

\title{Three Quantum Latin Squares of Order 6 with Cardinalities 13, 15, and 17}
\author[1]{Zhipeng Xu\thanks{Corresponding author: \texttt{xuzhp@ntu.edu.cn}}}
\affil[1]{School of Mathematics and Statistics, Nantong University, Nantong, China}

\begin{document}
\maketitle

\begin{abstract}
We give three explicit quantum Latin squares of order $6$, with cardinalities
$13$, $15$, and $17$. Throughout, vectors differing only by a global phase are
counted as identical. The cardinality-$13$ construction is based on an
orthogonal direct-sum decomposition $\C^6=\C^4\oplus\C^2$. The cardinality-$15$
and cardinality-$17$ constructions are based on two-dimensional Hadamard pairs
supported on coordinate planes.
\end{abstract}
\section{Definitions and background}

Classical Latin squares go back at least to Euler's work on Graeco-Latin
squares \cite{Euler1782}; standard background can be found in the monographs and
surveys on Latin squares and combinatorial designs
\cite{DenesKeedwell1974,ColbournDinitz2007}. The classical existence theory for
mutually orthogonal Latin squares includes the disproof of Euler's conjecture by
Bose, Shrikhande, and Parker \cite{BoseShrikhandeParker1960}. Quantum Latin
squares were introduced by Musto and Vicary \cite{MustoVicary2016} as a quantum
analogue of Latin squares, motivated in part by constructions of unitary error
bases and related quantum-information protocols \cite{Werner2001,TadejZyczkowski2006}.
Subsequent work developed notions of orthogonality and quantum combinatorial
designs \cite{GoyenecheRaissiMartinoZyczkowski2018,MustoVicary2019}, including
order-six quantum constructions related to Euler's 36 officers problem
\cite{RatherBurchardtBruzdaRajchelLakshminarayanZyczkowski2022,ZyczkowskiBruzdaRajchelEtAl2023}.
The cardinality convention used below follows recent work on the cardinalities
of quantum Latin squares \cite{ZhangWangJi2025,ZangZhengTianShan2025,ZhangJi2026};
closely related maximal-cardinality questions are studied in \cite{ZhangCao2026}.

\begin{definition}[Quantum Latin square]
A quantum Latin square of order $n$, denoted by $\QLS(n)$, is an $n\times n$
array
\[
  \Phi=(\ket{\phi_{ij}})_{0\leq i,j\leq n-1}
\]
whose entries are unit vectors in $\C^n$, such that every row and every column
of $\Phi$ forms an orthonormal basis of $\C^n$.
\end{definition}

Two unit vectors are identified if they differ only by a global phase. That is,
\[
  \ket{u}\sim \ket{v}
  \quad\Longleftrightarrow\quad
  \ket{u}=e^{\mathrm i\theta}\ket{v}
  \quad\text{for some }\theta\in\mathbb R.
\]
The cardinality of a quantum Latin square $\Phi$ is the number of distinct
vectors appearing in the array, modulo this equivalence relation:
\[
  \card(\Phi)=\#\{[\ket{\phi_{ij}}]:0\leq i,j\leq n-1\}.
\]

\section{A QLS(6) of cardinality 13}

Let
\[
  \ket{0},\ket{1},\ket{2},\ket{3},\ket{4},\ket{5}
\]
be the computational basis of $\C^6$. Decompose
\[
  \C^6=U\oplus V,
\]
where
\[
  U=\operatorname{span}\{\ket{0},\ket{1},\ket{2},\ket{3}\},
  \qquad
  V=\operatorname{span}\{\ket{4},\ket{5}\}.
\]

In $U$, define
\[
  a=\ket{0},\qquad b=\ket{1},\qquad c=\ket{2},\qquad d=\ket{3}.
\]
Set
\[
  e=\frac{c+d}{\sqrt2},
  \qquad
  v=\frac{c-d}{\sqrt2}.
\]
Now define
\[
  f=\frac{b+v}{\sqrt2}
   =\frac{\ket{1}}{\sqrt2}+\frac{\ket{2}-\ket{3}}{2},
  \qquad
  g=\frac{b-v}{\sqrt2}
   =\frac{\ket{1}}{\sqrt2}-\frac{\ket{2}-\ket{3}}{2},
\]
and
\[
  h=\frac{a+v}{\sqrt2}
   =\frac{\ket{0}}{\sqrt2}+\frac{\ket{2}-\ket{3}}{2},
  \qquad
  r=\frac{a-v}{\sqrt2}
   =\frac{\ket{0}}{\sqrt2}-\frac{\ket{2}-\ket{3}}{2}.
\]

In $V$, define
\[
  x=\ket{4},
  \qquad
  y=\ket{5},
\]
and
\[
  z=\frac{x+y}{\sqrt2},
  \qquad
  w=\frac{x-y}{\sqrt2}.
\]

For reference, the entries used in this construction have the following
coordinate representations with respect to the ordered computational basis
$(\ket{0},\ket{1},\ket{2},\ket{3},\ket{4},\ket{5})$:
\[
\small
\begin{array}{lll}
 a=(1,0,0,0,0,0)^{\mathsf T},
& b=(0,1,0,0,0,0)^{\mathsf T},
& c=(0,0,1,0,0,0)^{\mathsf T},\\[2mm]
 d=(0,0,0,1,0,0)^{\mathsf T},
& e=(0,0,\frac1{\sqrt2},\frac1{\sqrt2},0,0)^{\mathsf T},
& f=(0,\frac1{\sqrt2},\frac12,-\frac12,0,0)^{\mathsf T},\\[2mm]
 g=(0,\frac1{\sqrt2},-\frac12,\frac12,0,0)^{\mathsf T},
& h=(\frac1{\sqrt2},0,\frac12,-\frac12,0,0)^{\mathsf T},
& r=(\frac1{\sqrt2},0,-\frac12,\frac12,0,0)^{\mathsf T},\\[2mm]
 x=(0,0,0,0,1,0)^{\mathsf T},
& y=(0,0,0,0,0,1)^{\mathsf T},
& z=(0,0,0,0,\frac1{\sqrt2},\frac1{\sqrt2})^{\mathsf T},\\[2mm]
 w=(0,0,0,0,\frac1{\sqrt2},-\frac1{\sqrt2})^{\mathsf T}.&&
\end{array}
\]

Consider the following $6\times6$ array:
\[
\Phi_{13}=
\begin{pmatrix}
 c&a&x&y&b&d\\
 d&w&a&b&z&c\\
 y&f&g&x&e&a\\
 a&g&f&e&w&z\\
 x&e&y&h&r&b\\
 b&z&e&r&h&w
\end{pmatrix}.
\]

\begin{proposition}
The array $\Phi_{13}$ is a $\QLS(6)$ with cardinality $13$.
\end{proposition}

\begin{proof}
Define three orthonormal bases of $U$ by
\[
  B_0=\{a,b,c,d\},
  \qquad
  B_1=\{a,e,f,g\},
  \qquad
  B_2=\{b,e,h,r\}.
\]
The set $B_0$ is the computational basis of $U$. Since
\[
  e=\frac{c+d}{\sqrt2},
  \qquad
  v=\frac{c-d}{\sqrt2},
\]
the pair $\{e,v\}$ is an orthonormal basis of $\operatorname{span}\{c,d\}$.
The pair $\{f,g\}$ is obtained from $\{b,v\}$ by the Hadamard rotation
\[
  f=\frac{b+v}{\sqrt2},
  \qquad
  g=\frac{b-v}{\sqrt2}.
\]
Therefore $\{f,g\}$ is an orthonormal basis of $\operatorname{span}\{b,v\}$,
and it is orthogonal to both $a$ and $e$. Hence $B_1$ is an orthonormal basis of
$U$.

Similarly, the pair $\{h,r\}$ is obtained from $\{a,v\}$ by the Hadamard
rotation
\[
  h=\frac{a+v}{\sqrt2},
  \qquad
  r=\frac{a-v}{\sqrt2}.
\]
Thus $\{h,r\}$ is an orthonormal basis of $\operatorname{span}\{a,v\}$, and it
is orthogonal to both $b$ and $e$. Hence $B_2$ is also an orthonormal basis of
$U$.

Define two orthonormal bases of $V$ by
\[
  S_0=\{x,y\},
  \qquad
  S_1=\{z,w\}.
\]
The set $S_0$ is the computational basis of $V$, and $S_1$ is obtained from
$S_0$ by a two-dimensional Hadamard rotation. Since $U\perp V$, every union
$B_i\cup S_j$ is an orthonormal basis of $\C^6=U\oplus V$.

The six rows of $\Phi_{13}$ are of the following types:
\[
\begin{array}{c|c}
\text{row} & \text{orthonormal basis} \\\hline
R_1 & B_0\cup S_0\\
R_2 & B_0\cup S_1\\
R_3 & B_1\cup S_0\\
R_4 & B_1\cup S_1\\
R_5 & B_2\cup S_0\\
R_6 & B_2\cup S_1
\end{array}
\]
and the six columns are of the following types:
\[
\begin{array}{c|c}
\text{column} & \text{orthonormal basis} \\\hline
C_1 & B_0\cup S_0\\
C_2 & B_1\cup S_1\\
C_3 & B_1\cup S_0\\
C_4 & B_2\cup S_0\\
C_5 & B_2\cup S_1\\
C_6 & B_0\cup S_1
\end{array}
\]
Thus every row and every column is an orthonormal basis of $\C^6$, so
$\Phi_{13}$ is a $\QLS(6)$.

The vectors appearing in $\Phi_{13}$ are exactly
\[
  \vectset(\Phi_{13})=\{a,b,c,d,e,f,g,h,r,x,y,z,w\}.
\]
This set has $13$ elements. No two of them differ only by a global phase:
vectors from $U$ and vectors from $V$ have orthogonal supports; the four vectors
$x,y,z,w$ in $V$ are pairwise not scalar multiples; and the nine vectors
$a,b,c,d,e,f,g,h,r$ in $U$ have different supports or different relative signs
on a common support. Therefore
\[
  \card(\Phi_{13})=13.
\]
\end{proof}

\section{A QLS(6) of cardinality 15}

For later comparison with the cardinality-$17$ example, we first give a
coordinate-plane construction of cardinality $15$. Let
\[
  e_i=\ket{i},\qquad 0\leq i\leq 5.
\]
For $0\leq i<j\leq 5$, write
\[
  p_{ij}=\frac{e_i+e_j}{\sqrt2},
  \qquad
  q_{ij}=\frac{e_i-e_j}{\sqrt2}.
\]
Thus $\{p_{ij},q_{ij}\}$ is an orthonormal basis of the coordinate plane
$\operatorname{span}\{e_i,e_j\}$.

We use the five coordinate pairs
\[
  (0,1),\qquad (0,4),\qquad (1,2),\qquad (2,3),\qquad (2,4),
\]
and the five computational basis vectors
\[
  e_0,\qquad e_1,\qquad e_3,\qquad e_4,\qquad e_5.
\]
The complete list of vectors used in the construction is therefore
\[
\begin{gathered}
  p_{01},q_{01},\quad p_{04},q_{04},\quad p_{12},q_{12},\quad
  p_{23},q_{23},\quad p_{24},q_{24},\\
  e_0,e_1,e_3,e_4,e_5.
\end{gathered}
\]

Equivalently, the fifteen vectors used in $\Phi_{15}$ have the following
coordinate forms:
\[
\small
\begin{array}{lll}
 p_{01}=(\frac1{\sqrt2},\frac1{\sqrt2},0,0,0,0)^{\mathsf T},
& q_{01}=(\frac1{\sqrt2},-\frac1{\sqrt2},0,0,0,0)^{\mathsf T},
& p_{04}=(\frac1{\sqrt2},0,0,0,\frac1{\sqrt2},0)^{\mathsf T},\\[2mm]
 q_{04}=(\frac1{\sqrt2},0,0,0,-\frac1{\sqrt2},0)^{\mathsf T},
& p_{12}=(0,\frac1{\sqrt2},\frac1{\sqrt2},0,0,0)^{\mathsf T},
& q_{12}=(0,\frac1{\sqrt2},-\frac1{\sqrt2},0,0,0)^{\mathsf T},\\[2mm]
 p_{23}=(0,0,\frac1{\sqrt2},\frac1{\sqrt2},0,0)^{\mathsf T},
& q_{23}=(0,0,\frac1{\sqrt2},-\frac1{\sqrt2},0,0)^{\mathsf T},
& p_{24}=(0,0,\frac1{\sqrt2},0,\frac1{\sqrt2},0)^{\mathsf T},\\[2mm]
 q_{24}=(0,0,\frac1{\sqrt2},0,-\frac1{\sqrt2},0)^{\mathsf T},
& e_0=(1,0,0,0,0,0)^{\mathsf T},
& e_1=(0,1,0,0,0,0)^{\mathsf T},\\[2mm]
 e_3=(0,0,0,1,0,0)^{\mathsf T},
& e_4=(0,0,0,0,1,0)^{\mathsf T},
& e_5=(0,0,0,0,0,1)^{\mathsf T}.
\end{array}
\]

Consider the following $6\times6$ array:
\[
\Phi_{15}=
\begin{pmatrix}
 p_{04}&p_{12}&q_{12}&e_3&q_{04}&e_5\\
 p_{23}&e_5&e_4&p_{01}&q_{23}&q_{01}\\
 e_1&e_0&e_3&p_{24}&e_5&q_{24}\\
 q_{04}&q_{12}&p_{12}&e_5&p_{04}&e_3\\
 q_{23}&e_4&e_5&q_{01}&p_{23}&p_{01}\\
 e_5&e_3&e_0&q_{24}&e_1&p_{24}
\end{pmatrix}.
\]

\begin{proposition}
The array $\Phi_{15}$ is a $\QLS(6)$ with cardinality $15$.
\end{proposition}

\begin{proof}
Each row and each column is a disjoint union of coordinate-plane Hadamard pairs
and computational basis vectors. In the following table, $(ij)$ denotes the pair
$\{p_{ij},q_{ij}\}$ and $\{k\}$ denotes the singleton $\{e_k\}$.
The six rows have the decompositions
\[
\begin{array}{c|c}
\text{row} & \text{coordinate decomposition} \\\hline
R_1 & (04)\cup(12)\cup\{3\}\cup\{5\}\\
R_2 & (23)\cup(01)\cup\{4\}\cup\{5\}\\
R_3 & (24)\cup\{0\}\cup\{1\}\cup\{3\}\cup\{5\}\\
R_4 & (04)\cup(12)\cup\{3\}\cup\{5\}\\
R_5 & (23)\cup(01)\cup\{4\}\cup\{5\}\\
R_6 & (24)\cup\{0\}\cup\{1\}\cup\{3\}\cup\{5\}
\end{array}
\]
and the six columns have the decompositions
\[
\begin{array}{c|c}
\text{column} & \text{coordinate decomposition} \\\hline
C_1 & (04)\cup(23)\cup\{1\}\cup\{5\}\\
C_2 & (12)\cup\{0\}\cup\{3\}\cup\{4\}\cup\{5\}\\
C_3 & (12)\cup\{0\}\cup\{3\}\cup\{4\}\cup\{5\}\\
C_4 & (01)\cup(24)\cup\{3\}\cup\{5\}\\
C_5 & (04)\cup(23)\cup\{1\}\cup\{5\}\\
C_6 & (01)\cup(24)\cup\{3\}\cup\{5\}.
\end{array}
\]
In each row and in each column, the displayed coordinate pieces are mutually
orthogonal and together cover the coordinate set $\{0,1,2,3,4,5\}$. Hence each
row and each column is an orthonormal basis of $\C^6$. Thus $\Phi_{15}$ is a
$\QLS(6)$.

The vectors appearing in $\Phi_{15}$ are exactly
\[
\begin{gathered}
  p_{01},q_{01},\quad p_{04},q_{04},\quad p_{12},q_{12},\quad
  p_{23},q_{23},\quad p_{24},q_{24},\\
  e_0,e_1,e_3,e_4,e_5.
\end{gathered}
\]
The five coordinate pairs contribute $10$ Hadamard-pair vectors, and the five
listed computational basis vectors also occur. No Hadamard-pair vector is a
scalar multiple of a computational basis vector, and two Hadamard-pair vectors
from different coordinate pairs have different supports. For a fixed coordinate
pair $(i,j)$, the vectors $p_{ij}$ and $q_{ij}$ have different relative signs,
so they are not scalar multiples. Therefore the number of phase-classes is
\[
  10+5=15.
\]
Hence $\operatorname{card}(\Phi_{15})=15$.
\end{proof}

\section{A QLS(6) of cardinality 17}

For $0\leq i<j\leq 5$, define the two Hadamard-pair vectors
\[
  p_{ij}=\frac{\ket{i}+\ket{j}}{\sqrt2},
  \qquad
  q_{ij}=\frac{\ket{i}-\ket{j}}{\sqrt2}.
\]
The pair $\{p_{ij},q_{ij}\}$ is an orthonormal basis of the coordinate plane
$\operatorname{span}\{\ket{i},\ket{j}\}$.

We use the seven coordinate pairs
\[
  (0,1),\quad (0,3),\quad (0,4),\quad (1,2),
  \quad (1,5),\quad (2,3),\quad (2,4).
\]

For completeness, the seventeen vectors used in $\Phi_{17}$ are listed below
in coordinates:
\[
\small
\begin{array}{lll}
 p_{01}=(\frac1{\sqrt2},\frac1{\sqrt2},0,0,0,0)^{\mathsf T},
& q_{01}=(\frac1{\sqrt2},-\frac1{\sqrt2},0,0,0,0)^{\mathsf T},
& p_{03}=(\frac1{\sqrt2},0,0,\frac1{\sqrt2},0,0)^{\mathsf T},\\[2mm]
 q_{03}=(\frac1{\sqrt2},0,0,-\frac1{\sqrt2},0,0)^{\mathsf T},
& p_{04}=(\frac1{\sqrt2},0,0,0,\frac1{\sqrt2},0)^{\mathsf T},
& q_{04}=(\frac1{\sqrt2},0,0,0,-\frac1{\sqrt2},0)^{\mathsf T},\\[2mm]
 p_{12}=(0,\frac1{\sqrt2},\frac1{\sqrt2},0,0,0)^{\mathsf T},
& q_{12}=(0,\frac1{\sqrt2},-\frac1{\sqrt2},0,0,0)^{\mathsf T},
& p_{15}=(0,\frac1{\sqrt2},0,0,0,\frac1{\sqrt2})^{\mathsf T},\\[2mm]
 q_{15}=(0,\frac1{\sqrt2},0,0,0,-\frac1{\sqrt2})^{\mathsf T},
& p_{23}=(0,0,\frac1{\sqrt2},\frac1{\sqrt2},0,0)^{\mathsf T},
& q_{23}=(0,0,\frac1{\sqrt2},-\frac1{\sqrt2},0,0)^{\mathsf T},\\[2mm]
 p_{24}=(0,0,\frac1{\sqrt2},0,\frac1{\sqrt2},0)^{\mathsf T},
& q_{24}=(0,0,\frac1{\sqrt2},0,-\frac1{\sqrt2},0)^{\mathsf T},
& e_3=(0,0,0,1,0,0)^{\mathsf T},\\[2mm]
 e_4=(0,0,0,0,1,0)^{\mathsf T},
& e_5=(0,0,0,0,0,1)^{\mathsf T}.&
\end{array}
\]

Consider the following $6\times6$ array:
\[
\Phi_{17}=
\begin{pmatrix}
 p_{04}&p_{12}&q_{12}&\ket{3}&q_{04}&\ket{5}\\
 p_{23}&\ket{5}&\ket{4}&p_{01}&q_{23}&q_{01}\\
 p_{15}&p_{03}&q_{03}&p_{24}&q_{15}&q_{24}\\
 q_{04}&q_{12}&p_{12}&\ket{5}&p_{04}&\ket{3}\\
 q_{23}&\ket{4}&\ket{5}&q_{01}&p_{23}&p_{01}\\
 q_{15}&q_{03}&p_{03}&q_{24}&p_{15}&p_{24}
\end{pmatrix}.
\]

\begin{proposition}
The array $\Phi_{17}$ is a $\QLS(6)$ with cardinality $17$.
\end{proposition}

\begin{proof}
Each row and column of $\Phi_{17}$ is a disjoint union of two-dimensional
Hadamard pairs, possibly together with leftover computational basis vectors.
More explicitly, the six rows have the following decompositions:
\[
\begin{array}{c|c}
\text{row} & \text{coordinate decomposition} \\\hline
R_1 & (04)\cup(12)\cup\{3\}\cup\{5\}\\
R_2 & (23)\cup(01)\cup\{4\}\cup\{5\}\\
R_3 & (15)\cup(03)\cup(24)\\
R_4 & (04)\cup(12)\cup\{3\}\cup\{5\}\\
R_5 & (23)\cup(01)\cup\{4\}\cup\{5\}\\
R_6 & (15)\cup(03)\cup(24)
\end{array}
\]
Here $(ij)$ means the orthonormal pair $\{p_{ij},q_{ij}\}$, while $\{k\}$ means
the singleton $\{\ket{k}\}$. In every row, the displayed coordinate pieces are
mutually disjoint and together cover $\{0,1,2,3,4,5\}$. Hence every row is an
orthonormal basis of $\C^6$.

The six columns have the following decompositions:
\[
\begin{array}{c|c}
\text{column} & \text{coordinate decomposition} \\\hline
C_1 & (04)\cup(23)\cup(15)\\
C_2 & (12)\cup(03)\cup\{4\}\cup\{5\}\\
C_3 & (12)\cup(03)\cup\{4\}\cup\{5\}\\
C_4 & (01)\cup(24)\cup\{3\}\cup\{5\}\\
C_5 & (04)\cup(23)\cup(15)\\
C_6 & (01)\cup(24)\cup\{3\}\cup\{5\}
\end{array}
\]
Again, in every column the displayed coordinate pieces are mutually disjoint and
together cover $\{0,1,2,3,4,5\}$. Hence every column is also an orthonormal basis
of $\C^6$. Thus $\Phi_{17}$ is a $\QLS(6)$.

It remains to compute its cardinality. The distinct Hadamard-pair vectors
appearing in $\Phi_{17}$ are
\[
\begin{gathered}
  p_{01},q_{01},\quad p_{03},q_{03},\quad p_{04},q_{04},\quad
  p_{12},q_{12},\\
  p_{15},q_{15},\quad p_{23},q_{23},\quad p_{24},q_{24}.
\end{gathered}
\]
These are $14$ vectors. In addition, the computational basis vectors
\[
  \ket{3},\qquad \ket{4},\qquad \ket{5}
\]
appear in the array. Therefore at most $17$ phase-classes appear.

No two of the listed vectors differ only by a global phase. Indeed, the vectors
$p_{ij}$ and $q_{ij}$ both have support $\{i,j\}$ but have different relative
signs, while vectors associated with different pairs have different supports.
No $p_{ij}$ or $q_{ij}$ is a scalar multiple of a computational basis vector,
because each has support of size $2$. Hence
\[
  \card(\Phi_{17})=14+3=17.
\]
\end{proof}

\begin{remark}
The three examples use two related mechanisms. The cardinality-$13$ example is a
$4+2$ direct-sum construction with shared bases in the four-dimensional part.
The cardinality-$15$ example is a coordinate-plane construction in which five
Hadamard pairs contribute $10$ vectors and five computational basis vectors
remain present, giving $10+5=15$. The cardinality-$17$ example uses seven
Hadamard pairs, contributing $14$ vectors, together with three computational
basis vectors, giving $14+3=17$.
\end{remark}
\section*{Acknowledgments}

The author gratefully acknowledge support from the Natural Science Research Project of Jiangsu Higher Education Institutions of China under Grant No. 24KJB520033. The author acknowledge the High Performance Computing Center of the School of Mathematics and Statistics, Nantong University, together with the National Supercomputing Center in Kunshan, for their support in providing computational resources.


\begin{thebibliography}{99}

\bibitem{Euler1782}
L. Euler,
\newblock \emph{Recherches sur une nouvelle esp\`ece de quarr\'es magiques},
\newblock Verhandelingen uitgegeven door het Zeeuwsch Genootschap der Wetenschappen te Vlissingen
\textbf{9}, 85--239, 1782.

\bibitem{DenesKeedwell1974}
J. D\'enes and A. D. Keedwell,
\newblock \emph{Latin Squares and Their Applications},
\newblock Akad\'emiai Kiad\'o, Budapest, 1974.

\bibitem{ColbournDinitz2007}
C. J. Colbourn and J. H. Dinitz, editors,
\newblock \emph{Handbook of Combinatorial Designs}, second edition,
\newblock Chapman \& Hall/CRC, Boca Raton, 2007.

\bibitem{BoseShrikhandeParker1960}
R. C. Bose, S. S. Shrikhande, and E. T. Parker,
\newblock \emph{Further results on the construction of mutually orthogonal Latin squares and the falsity of Euler's conjecture},
\newblock Canadian Journal of Mathematics \textbf{12}, 189--203, 1960.
\newblock doi:\href{https://doi.org/10.4153/CJM-1960-016-5}{10.4153/CJM-1960-016-5}.

\bibitem{Werner2001}
R. F. Werner,
\newblock \emph{All teleportation and dense coding schemes},
\newblock Journal of Physics A: Mathematical and General \textbf{34}(35), 7081--7094, 2001.
\newblock arXiv:quant-ph/0003070, doi:\href{https://doi.org/10.1088/0305-4470/34/35/332}{10.1088/0305-4470/34/35/332}.

\bibitem{TadejZyczkowski2006}
W. Tadej and K. \.{Z}yczkowski,
\newblock \emph{A concise guide to complex Hadamard matrices},
\newblock Open Systems \& Information Dynamics \textbf{13}(2), 133--177, 2006.
\newblock arXiv:quant-ph/0512154, doi:\href{https://doi.org/10.1007/s11080-006-8220-2}{10.1007/s11080-006-8220-2}.

\bibitem{MustoVicary2016}
B. Musto and J. Vicary,
\newblock \emph{Quantum Latin squares and unitary error bases},
\newblock Quantum Information and Computation \textbf{16}(15--16), 1318--1332, 2016.
\newblock arXiv:1504.02715.

\bibitem{GoyenecheRaissiMartinoZyczkowski2018}
D. Goyeneche, Z. Raissi, S. Di Martino, and K. \.{Z}yczkowski,
\newblock \emph{Entanglement and quantum combinatorial designs},
\newblock Physical Review A \textbf{97}, 062326, 2018.
\newblock arXiv:1708.05946, doi:\href{https://doi.org/10.1103/PhysRevA.97.062326}{10.1103/PhysRevA.97.062326}.

\bibitem{MustoVicary2019}
B. Musto and J. Vicary,
\newblock \emph{Orthogonality for quantum Latin isometry squares},
\newblock Electronic Proceedings in Theoretical Computer Science \textbf{287}, 253--266, 2019.
\newblock arXiv:1804.04042, doi:\href{https://doi.org/10.4204/EPTCS.287.15}{10.4204/EPTCS.287.15}.

\bibitem{RatherBurchardtBruzdaRajchelLakshminarayanZyczkowski2022}
S. A. Rather, A. Burchardt, W. Bruzda, G. Rajchel-Mieldzio\'c, A. Lakshminarayan, and K. \.{Z}yczkowski,
\newblock \emph{Thirty-six entangled officers of Euler: Quantum solution to a classically impossible problem},
\newblock Physical Review Letters \textbf{128}, 080507, 2022.
\newblock arXiv:2104.05122, doi:\href{https://doi.org/10.1103/PhysRevLett.128.080507}{10.1103/PhysRevLett.128.080507}.

\bibitem{ZyczkowskiBruzdaRajchelEtAl2023}
K. \.{Z}yczkowski, W. Bruzda, G. Rajchel-Mieldzio\'c, A. Burchardt, S. A. Rather, and A. Lakshminarayan,
\newblock \emph{$9\times4=6\times6$: Understanding the quantum solution to Euler's problem of 36 officers},
\newblock Journal of Physics: Conference Series \textbf{2448}, 012003, 2023.
\newblock arXiv:2204.06800, doi:\href{https://doi.org/10.1088/1742-6596/2448/1/012003}{10.1088/1742-6596/2448/1/012003}.

\bibitem{ZhangWangJi2025}
Y. Zhang, X. Wang, and L. Ji,
\newblock \emph{Quantum Latin squares with all possible cardinalities},
\newblock arXiv:2507.05642, 2025.

\bibitem{ZangZhengTianShan2025}
Y. Zang, M. Zheng, Z. Tian, and X. Shan,
\newblock \emph{On the cardinalities of quantum Latin squares},
\newblock arXiv:2508.01972, 2025.

\bibitem{ZhangJi2026}
Y. Zhang and L. Ji,
\newblock \emph{Quantum Latin squares of order $6m$ with all possible cardinalities},
\newblock arXiv:2601.09132, 2026.

\bibitem{ZhangCao2026}
Y. Zhang and H. Cao,
\newblock \emph{The maximal cardinality of quantum Latin squares},
\newblock Discrete Mathematics \textbf{349}, 114863, 2026.

\end{thebibliography}
\end{document}